\documentclass{amsart}
\usepackage{amssymb}   % For Latex2e
\usepackage{amsmath}
\usepackage{amsthm}
\usepackage[all,dvips]{xy}
\def\a={{\buildrel a \over =}}
\def\na={{\buildrel a \over \neq}}
\def\CC{{\mathbb C}}
\def\PP{{\mathbb P}}
\def\Fcal{{\mathcal F}}
\def\Lcal{{\mathcal L}}
\def\Pcal{{\mathcal P}}
\def\Ocal{{\mathcal O}}
\def\Ycal{{\mathcal Y}}
\def\Xcal{{\mathcal X}}
\def\GG{{\mathbb G}}
\def\QQ{{\mathbb Q}}
\def\RR{{\mathbb R}}
\def\ZZ{{\mathbb Z}}

\def\X{{\mathcal X}}
\def\Y{{\mathcal Y}}
\def\C{{\mathcal C}}
\def\Z{{\mathcal Z}}

\def\Xs{{\Xcal}^{\star}}

\newtheorem{theorem}{Theorem}[section]
\newtheorem{lemma}[theorem]{Lemma}
\newtheorem{proposition}[theorem]{Proposition}
\newtheorem{corollary}[theorem]{Corollary}

\newtheorem{definition-lemma}[theorem]{Definition-Lemma}

\theoremstyle{definition}

\theoremstyle{remark}

\begin{document}

\title[The Rank-One Limit of the Fourier-Mukai Transform]
{The Rank-One Limit of the Fourier-Mukai Transform}
\author{Gerard van der Geer $\,$}
\address{Korteweg-de Vries Instituut, Universiteit van
Amsterdam, Postbus 94248, 1090 GE Amsterdam,  The Netherlands}
\email{G.B.M.vanderGeer@uva.nl}
\author{$\,$ Alexis Kouvidakis}
\address{Department of Mathematics, University of Crete, 
GR-71409 Heraklion, Greece}
\email{kouvid@math.uoc.gr}
\begin{abstract}
We give a formula for the specialization of the Fourier-Mukai
transform on a semi-abelian variety of torus rank $1$.
\end{abstract}

\maketitle

\begin{section}{Introduction}\label{sec:intro}
Let $\pi: {\Xs}\to S$ be a semi-abelian variety of 
relative dimension $g$ over the spectrum $S$ of a 
discrete valuation ring $R$ with algebraically closed residue field $k$
such that the generic fibre $X_{\eta}$ is a principally polarized
abelian variety. We assume that ${\Xs}$ is contained in a
complete rank-one degeneration ${\X}$. In particular,
the special fibre $X_0$ of ${\X}$ is a complete variety over $k$ containing
as an open part the total space of the $\GG_m$-bundle associated to a line bundle $J\to B$
over a $g-1$-dimensional abelian variety~$B$. The normalization
$\nu: {\PP}\to X_0$ of $X_0$ can be identified with the 
${\PP}^1$-bundle over $B$ associated to $J$ and $X_0$ is 
obtained by identifying the zero-section
of ${\PP}\cong B$ with the infinity-section of ${\PP}$ by a translation.
Moreover, $X_0$ is provided with a theta divisor that is the specialization
of the polarization divisor on the generic fibre.

If $c_{\eta}$ is an algebraic cycle on $X_{\eta}$  we can 
 take the Fourier-Mukai transform
$\varphi_{\eta}:=F(c_{\eta})$ and consider the limit cycle (specialization)
$\varphi_0$ of $\varphi_{\eta}$. A natural question is: 
What is the limit $\varphi_0$ of $\varphi_{\eta}$?

If $q:{\PP}\to B$ denotes the natural projection of the ${\PP}^1$-bundle, 
the Chow ring of ${\PP}$ is the extension ${\rm CH}^*(B)[\eta]/(\eta^2-\eta \cdot q^*c_1(J))$
with $\eta =c_1 (O_{\PP}(1))$. We consider now
cycles with rational coefficients.
We denote by $c_0$ the specialization of the cycle $c_{\eta}$ on $X_0$. 
We can write  $c_0$ as $\nu_*(\gamma)$ with $\gamma
=  q^*z+ q^*w\cdot \eta$.

\begin{theorem} \label{th: aeFT}
 Let $c_{\eta}$ be a cycle on $X_{\eta}$ with  
$c_0=\nu_*(q^*z+q^*w\cdot \eta)$.
The limit $\varphi_0$ 
of the Fourier-Mukai  transform $\varphi_{\eta}=F(c_{\eta})$ 
is given by  $\varphi_0=\nu_*(  q^*a+q^* b \cdot \eta)$
with 
$$
a= F_B(w)+ \sum_{n=0}^{2g-2}\sum_{m=0}^{n} \frac{(-1)^m}{(n+2)!}\, 
F_B[( z+w\cdot c_1(J)) \cdot  c_1^{m}(J)] \cdot c_1^{n-m+1}(J)
$$
and
$$
b= \sum_{n=0}^{2g-2}\sum_{m=0}^{n} \frac{(-1)^{m}}{(n+2)!}\,  
F_B[(( (-1)^{n+1}-1)z - w\cdot c_1(J))  
 \,  \cdot c_1^{m}(J)] \cdot c_1^{n-m}(J) \,   ,
$$
where $F_B$ is the Fourier-Mukai transform of the abelian variety $B$.
\end{theorem}

We denote algebraic equivalence by $\a=$. The relation $c_1(J)\a= 0$ implies
the following result.

\begin{theorem} \label{cor: aeFT}
With the above notation the limit $\varphi_0$ satisfies
$$
\varphi_0 \, \a= \, \nu_*(q^* F_B(w)-q^*F_B(z)\cdot \eta)\, .
$$
\end{theorem}
Note that this is compatible with the fact that for a
principally polarized abelian variety $A$ of dimension $g$ 
the Fourier-Mukai transform satisfies 
$F_{A} \circ F_A= (-1)^g (-1_A)^*$. 

Beauville introduced in \cite{B1} a decomposition on the Chow
ring with rational coefficients of an abelian variety using
the Fourier-Mukai transform. Theorem \ref{cor: aeFT}
can be used to deduce non-vanishing results for Beauville
components of cycles on the generic fibre of a semi-abelian variety
of rank $1$; we refer to \S \ref{Applications} for examples.

We prove the theorem by constructing a smooth model ${\Y}$ 
of ${\X}\times_{S}{\X}$ to which the addition map 
${\Xs}\times_{S}{\Xs} \to {\Xs}$ extends and by 
choosing an appropriate extension of the Poincar\'e bundle to ${\Y}$.
The proof is then reduced to a calculation in the special fibre.
We refer to Fulton's book \cite{F} for the intersection theory we use. 
The theory in that book is built for algebraic schemes over a field. 
In our case we work over the spectrum  of a 
discrete valuation ring. But as is
stated in \S \, 20.1 and 20.2 there, most of the theory in Fulton's book, 
including in particular the statements we use in this paper, is valid 
for schemes of finite type and separated over $S$. However, for us 
projective space denotes the space of hyperplanes and not lines, which
conflicts with Fulton's book, but is in accordance with \cite{Ha}. 
\end{section}
%%%%%%%%%%%%%%%%%%%%%%%%%%%%%%%%%%%%%%%%%%%%%%%%%%%%%%%%%%%%%%%%%%%%%%%%%%%%%%%%%%
%%%%%%%%%%%%%%%%%%%%%%%%%%%%%%%%%%%%%%%%%%%%%%%%%%%%%%%%%%%%%%%%%%%%%%%%%%%%%%%%%%
\begin{section}{Families of abelian varieties with a rank one degeneration }
\label{sec:family}\label{families}
We now assume that $R$ is a complete 
discrete valuation ring with local parameter~$t$,
field of quotients $K$ and algebraically closed residue field $k$.
Suppose that $({\Xs},{\mathcal L}) $ is a semi-abelian variety 
over $S={\rm Spec}(R)$ 
such that the generic fibre ${X}_{\eta}$ is abelian and the special fibre 
$X^{\ast}_0$ has torus rank $1$; moreover, we assume that ${\mathcal L}$ 
is a cubical 
invertible sheaf (meaning that ${\mathcal L}$ satisfies the theorem of the cube, 
see \cite{FC}, p. 2, 8) and $L_{\eta}$ is ample.
In particular, the special fibre of ${\Xs}$ fits in an exact sequence
$$
1 \to T_0 \to X_0^{\ast} \to B \to 0,
$$
where $B$ is an abelian variety over $k$ and $T_0$ the multiplicative 
group ${\GG}_m$ over $k$. The torus $T_0$ lifts uniquely to a torus $T_i$ 
of rank $1$ over $S_i= {\rm Spec}(R/(t^{i+1})$ in $X_i^{\ast}= 
{\Xs} \times_S S_i$.
The quotient $X_i^{\ast}/T_i$ is an abelian variety $B_i$ over $S_i$. The system
$\{ B_i\}_{i=1}^{\infty}$ defines a formal abelian variety which is 
algebraizable, so that we have an exact sequence of group schemes over $S$
$$
1 \to T \to G \;{\buildrel \pi \over \to} \;{\mathcal B} \to 0,
$$
cf.\ [F-C, p.\ 34].
We assume now that we are given a line bundle $M$ on ${\mathcal B}$ 
defining a principal 
polarization $\lambda: B \to B^t$ and consider $\pi^*(M)$. This defines 
a cubical line bundle on~$G$. The extension $G$ is given by a homomorphism
$c$ of the character group $Z\cong {\ZZ}$ of $T$ to ${\mathcal B}^t$.
The semi-abelian group scheme dual to ${\Xs}$ defines a similar extension
$$
1 \to T^t \to G^t \to {\mathcal B}^t \to 0
$$
and the polarization provides an
isomorphism $\phi$ of the character group $Z$ of $T$ with the
character group $Z^t$ of $T^t$.
Now the degenerating abelian variety (i.e.\ semi-abelian variety) ${\Xs}$
over $S$ gives rise to the set of degeneration data 
(cf.\ \cite{FC}, p 51, Thm 6.2, or \cite{AN}, Def.\ 2.3):
\begin{enumerate}
\item[(i)] an abelian variety ${\mathcal B}$ over $S$ and a rank $1$ extension $G$. This amounts to a $S$-valued point $b$ of ${\mathcal B}={\mathcal B}^t$.
\item[(ii)] a $K$-valued point of $G$ lying over $b$.
\item[(iii)] a cubical ample sheaf $L$ on $G$ inducing 
the polarization on ${\mathcal B}$ and an action of $Z=Z^t$
on $L_{\eta}$.
\end{enumerate}
A section $s \in \Gamma(G,L)$ can be written uniquely as
$s = \sum_{\chi \in Z} \sigma_{\chi}(s)$,
where $\sigma_{\chi}: \Gamma(G,L) \to \Gamma({\mathcal B},M_{\chi})$ is a
$R$-linear homomorphism and $M_{\chi}$ is the twist of $M$ by $\chi$:
in fact $\pi_*(O_G)=\oplus_{\chi} O_{\chi}$ with $O_{\chi}$ the subsheaf
consisting of $\chi$-eigenfunctions. (We refer to \cite{FC},
p.\ 43; note also the sign conventions there in the last lines.) 
We have now by the action
$$
c^t(y)^*M\cong M_{\phi(y)}\cong M\otimes O_{\phi(y)}, \qquad y \in Z^t.
$$
This satisfies
$\sigma_{\chi+1}(s)=\psi(1) \tau(\chi) T^*_{b}(\sigma_{\chi}(s))$,
where $\tau$ is given by a point of $G(K)$ lying over $b$ and $\psi$
is as in \cite{FC}, p.\ 44.
We refer to Faltings-Chai's theorem (6.2) of \cite{FC}, p.\ 51
 for the degeneration data.

The compactification ${\X}$ of ${\Xs}$ is now constructed as a quotient of
the action of $Z^t$ on a so-called relatively complete model.
Such a relatively complete model $\tilde{P}$
for $G$ can be constructed here in an essentially unique way. 
If $B$ is trivial (i.e.\ $\dim(B)=0$)  and if the torus is 
$T={\rm Spec}(R[z,z^{-1}])$  it is given as the toroidal variety 
obtained by gluing the affine pieces
$$
U_n={\rm Spec}(R[x_n,y_n]), \qquad {\rm with} \qquad x_ny_n=t
$$
where $G \subset \tilde{P}$ is given by
$x_n=z/t^n, \quad y_n=t^{n+1}/z$,
cf.\ \cite{Mumford1}, also in \cite{FC}, p.\ 306].
By glueing we obtain an infinite chain $\tilde{P}_0$ of ${\PP}^1$'s
in the special fibre. We can `divide' by the action of $Z^t$;
this is easy in the analytic case, more involved in
the algebraic case, but amounts to the same, cf.\ \cite{Mumford1}, 
also \cite{FC}, p.\ 55-56.

In the special fibre we find a rational curve with one ordinary
double point. If instead we divide by the action of $nZ^t$ for $n>1$
we find a cycle consisting of $n$ copies of ${\PP}^1$.

In case the abelian part $B$ is not trivial we take as a relatively
complete model the contracted
(or smashed) product $\tilde{P}\times^T G$ with $\tilde{P}$ the 
relatively complete model for the case that $B$ is trivial. Call
the resulting space $\tilde{P}$.
Then $\tilde{P}$ corresponds by
Mumford's [loc.\ cit., p 29] to a polyhedral decomposition of 
$Z^t\otimes {\RR}={\RR}$ with $Z^t$ the cocharacter group of $T$. 
Then we essentially divide through the action of $Z^t$ or $nZ^t$ as before
and obtain a proper ${\X} \to S$.

We describe the central fibre $X_0$ of ${\X}$.
Let $b$ be the $k$-valued point of $B\cong B^t$ that
determines the above ${\GG}_m$-extension. 
If $M$ denotes a line bundle defining the principal polarization
of $B$ we let $M_b$ be the translation of $M$ by $b$
and we set $J=M \otimes M_b^{-1}$ and 
define the projective bundle ${\PP}={\PP}(J\oplus {\Ocal}_B)$
with projection $q:{\PP} \rightarrow B$.  
The bundle ${\PP}$ has two natural sections (with images) ${\PP}_1$ and 
${\PP}_2$ corresponding to the projections 
$J\oplus  {\Ocal}_B \rightarrow  J$ and $J\oplus {\Ocal}_B\rightarrow {\Ocal}_B$.
We have ${\Ocal}({\PP}_1)\cong {\Ocal}({\PP}_2) \otimes q^*J $
and ${\Ocal}(1)\cong  \Ocal ({\PP}_1)$ with ${\Ocal}(1)$ the natural
line bundle on ${\PP}$.  
We denote by  $\overline{\PP}$ the non-normal variety obtained by gluing  
the sections ${\PP}_1$ and ${\PP}_2$ under a translation by the point $b$. 
The singular locus of $\overline{\PP}$ has support isomorphic to $B$. 
The line bundle 
$\tilde{L}={\Ocal}({\PP}_1)\otimes q^*M_b \cong {\Ocal}({\PP}_2)
\otimes q^*M$ 
descends to a line bundle $\overline{L}$ on $\overline{\PP}$ 
with a unique ample divisor $D$, see \cite{Mu}.  
The central family $X_0$ of the  family $\pi: \Xcal \rightarrow S$ 
is then equal to $\bar{\PP}$. 
The cubical invertible sheaf $\Lcal $ on $\Xs$ extends (uniquely)
to $\X$ and its restriction to the central fiber $\bar{\PP}$ is the line bundle
 $\overline{L}$, see~\cite{Na}. 
\end{section}
%%%%%%%%%%%%%%%%%%%%%%%%%%%%%%%%%%%%%%%%%%%%%%%%%%%%%%%%%%%%%%%%%%%%%%%%%%%%%%%%%%%%%%%%%%
%%%%%%%%%%%%%%%%%%%%%%%%%%%%%%%%%%%%%%%%%%%%%%%%%%%%%%%%%%%%%%%%%%%%%%%%%%%%%%%%%%%%%%%%%%
\begin{section}{Extension of the addition map} \label{sec:addition}
The addition map $\mu: {\Xs}\times_S {\Xs} \to {\Xs}$ of the semi-abelian
scheme ${\Xs}$ does not extend to a morphism ${\X}\times_S {\X} \to {\X}$, 
but it does so after a small blow-up of ${\X}\times_S {\X}$ as we shall see. 

The degeneration data of ${\Xs}$ defines (product) degeneration data for 
${\Xs}\times_S {\Xs}$. Indeed, we can take the fibre product of
the relatively complete model
$\tilde{P}'=\tilde{P}\times_S \tilde{P}$ and this corresponds 
(e.g.\ via \cite{Mumford1}, Corollary (6.6)) to the
standard polyhedral decomposition of 
${\RR}^2=(Z^t\otimes {\RR})^2$ by the lines $x=m$ and
$y=n$ for $m,n \in {\ZZ}$.
The special fibre of the model $\tilde{P}'$ is an infinite union of
${\PP}^1 \times {\PP}^1$-bundles over $B\times B$ glued along 
the fibres over $0$ and $\infty$.
The compactified model of ${\X}\times_S {\X}$ is obtained by taking the `quotient'
of $\tilde{P}'$ under the action of $Z^t \times Z^t$. This is not regular;
for example the criterion of Mumford (\cite{Mumford1}, p.\ 29, point (D)])
is not satisfied. We can remedy this by subdividing.
For example, by taking the decomposition of ${\RR}^2$
given by the lines $x=m, y=n$ and $x+y=l$ for $m,n,l \in {\ZZ}$.

The special fibre of this model is  an infinite union of copies
of ${\PP}^1\times{\PP}^1$-bundles over $B \times B$  
blown up in the two anti-diagonal sections
$(0,\infty)=\PP_1\times \PP_2$ and $(\infty,0)=\PP_2\times \PP_1$.  
This is regular. 

Both the polyhedral decompositions are invariant under the action of
translations $(x,y) \mapsto (x+a,y+b)$ for fixed $a,b \in {\ZZ}$.
This means that we can form the `quotient' by $Z^t\times Z^t\cong {\ZZ}^2$ 
(or a subgroup $nZ^t \times nZ^t$) and obtain 
a completed semi-abelian abelian variety ${\Y}$ of relative dimension $2g$
over $S$. We denote by $\epsilon: {\Y}\to {\Y}'=\X \times_S \X$ the natural map. 
We shall write $V$ for $Y_0$ and $\sigma: \tilde{V}\to V$ 
for its normalization. Then $\tilde{V}$ is an irreducible component 
of the special fibre of $\tilde{P}'$. 
We denote by $\tau: \tilde{V} \to {\PP}^1\times{\PP}^1$ the blow up map 
and by $E_{12}$ and $E_{21}$ the exceptional divisors over the 
blowing up loci $\PP_1\times \PP_2$ and $\PP_2\times \PP_1$, respectively. 

Now consider the addition map $\mu: {\Xs} \times_S {\Xs} \to {\Xs}$ 
with ${\Xs}$
as in the preceding section. This morphism is induces (and is induced by) 
by a map $\tilde{\mu}: G\times_S G \to  G$. However,
this map does not extend to a morphism  of the relatively complete
model $\tilde{P}'$ since the corresponding
(covariant) map $(Z^t\otimes {\RR})^2 \to (Z^t\otimes {\RR})$ does not
have the property that it maps cells to cells. After subdividing (by adding
the lines $x+y=l$ with $l \in {\ZZ}$)
this property is satisfied (cf.\  \cite{KKMS}, Thm.\ 7, p.\ 25).
This means that the map
$\mu$ extends to $\tilde{\mu}: \tilde{P}' \to \tilde{P}$ for  the
polyhedral decomposition given by   this subdivision. 
It is compatible with the action of
${\ZZ}$ and ${\ZZ}\times {\ZZ}$ 
and hence descends to a morphism $\bar{\mu}:  {\Ycal} \to {\Xcal}$.
We summarize:

\begin{proposition} \label{pro:extension}
The addition map of group schemes  $\mu : {\Xs}\times_S {\Xs}
 \rightarrow X^{\star}$ extends to a morphism
$\bar{\mu} : {\Y} \rightarrow {\X}$.
\end{proposition}

In the next section we shall see that the change from the model 
${\X}\times_S{\X}$ to ${\Y}$ is a small blow-up.

For later calculations we write down this map explicitly 
on the special fibre. We start with $g=1$; then $B$ is trivial and we may
restrict the
map to an irreducible component of the special
fibre of the relatively complete model $\tilde{P}\times_S\tilde{P}$
and get the map $m: {\PP}^1 \times {\PP}^1 \to {\PP}^1$ given by
$((a:b),(a':b')) \mapsto (aa':bb')$. This is not defined in the points
$(0,\infty)$ and $(\infty,0)$. After blowing up these points 
(which corresponds exactly to the change from 
${\X}\times_S{\X}$ to ${\Y}$) 
the rational map becomes a regular map $\tilde{m}: \tilde{V} \to {\PP}^1$. 
It is defined by the two sections 
${\rm prop}(p_1^* \{0\})+{\rm prop}(p_2^*\{ 0\} )$ and 
${\rm prop}(p_1^* \{\infty \})+{\rm prop}(p_2^*\{ \infty \} )$ of the
linear system $|\tau^*(F_1+F_2)-E_{12}-E_{21}|$
with $F_1$ and $F_2$ the horizontal and vertical fibre  
(with ${\rm  prop}(\;)$ meaning the proper transform). 
The map $\tilde{m}$ descends 
to a map $\bar{m} :V \to \bar{\PP}$ which is the restriction of the morphism 
$\bar{\mu} : {\Y} \rightarrow {\X}$ to the central fiber.  

For the case that $g>1$, note that we have the addition map $\mu_{\Xs}$.
Its restriction to the special fibre extends to a map
of the relatively complete model and then restricts to a morphism
$\tilde{m}: \tilde{V} \to {\PP}$ that lifts the addition map $\mu_B$ of $B$.
That means that it comes from a surjective 
bundle map (cf.\ \cite{Ha}, Ch.\ II, Prop.\ 7.12)
$$
\delta: m_1 ^*(J\oplus {\Ocal})\cong (p_1^*q^*J \otimes p_2^*q^*J) \oplus 
{\Ocal}  \to N
$$ 
with $m_1:=\mu_B \circ (q\times q) \circ \tau: \tilde{V} \to B$ 
and $N=\tau^*(p_1^*{\Ocal}({\PP}_1) \otimes p_2^*{\Ocal}({\PP}_1))
\otimes {\Ocal}(-E_{12}-E_{21}))$ with 
$p_i: \PP \times \PP \to \PP$ the $i$th projection. 
Then $m_1^*(J\oplus {\Ocal})^{\vee}\otimes N$ is isomorphic to
the direct sum of
$$
\tau^*p_1^*{\Ocal}({\PP}_i)\otimes \tau^*p_2^*{\Ocal}({\PP}_i)
\otimes {\Ocal}(-E_{12}-E_{21}) \qquad (i=1,2).
$$ 
The map $\delta $ is then given by the two sections
${\rm prop}(p_1^* {\PP}_i)+{\rm prop}(p_2^*{\PP}_i)$
of $\tau^*p_1^*{\Ocal}({\PP}_i)\otimes \tau^*p_2^*{\Ocal}({\PP}_i)
\otimes {\Ocal}(-E_{12}-E_{21})$ for $i=1,2$.
The map $\tilde{m}$ descends 
to a map $\bar{m} :V \to \bar{\PP}$ which is the restriction of the morphism 
$\bar{\mu} : {\Y} \rightarrow {\X}$ to the central fiber.
\end{section}
%%%%%%%%%%%%%%%%%%%%%%%%%%%%%%%%%%%%%%%%%%%%%%%%%%%%%%%%%%%%%%%%%%%%%%%%%%%%%%%%%%%%%%%%%%%%
%%%%%%%%%%%%%%%%%%%%%%%%%%%%%%%%%%%%%%%%%%%%%%%%%%%%%%%%%%%%%%%%%%%%%%%%%%%%%%%%%%%%%%%%%%%%
\begin{section}{An explicit model of ${\Y}$}
\label{modelY}
We now describe an explicit local construction of the model ${\Y}$ 
by blowing up the model ${\X}\times_S{\X}$. 
Let ${A}^{g+1}_S={\rm Spec}(R[x_1,\ldots,x_{g+1}])$ denote affine 
$S$-space. In local coordinates, inside ${A}_S^{g+1}$, we may assume 
that the $g$-dimensional fibration  $\pi: {\Xs} \to S$ is
given by the equation $x_1x_2=t$, where the coordinates 
$x_3, \ldots, x_{g+1}$ are not involved, see \cite{Mu} p.\ 361-362.
We may assume that the zero section of the family is defined  by 
$x_i=1$ for $i=1,\ldots,g+1$.

We form the fiber product $\pi: {\Y}'={\X}\times_S {\X}$. 
We denote by $T$ the support of the singular locus of $X_0$.
The $2g+1$ dimensional variety ${\Y}'$ is  singular in the special fiber 
along $\Sigma =T \times_k T \cong B \times_k B $ of dimension $2g-2$.
The generic fiber  $Y'_{\eta}$ is the product $X_{\eta}\times_KX_{\eta}$
of the abelian variety $X_{\eta}$, while the zero fiber $Y'_0$ is singular.  
The local equations of ${\Y}'$ in a neighborhood of the singular locus of 
the family are given in our local coordinates by the system  
$x_1x_2=t, \, x'_1x'_2=t$.
The singular locus $\Sigma$ of ${\Y}'$ is given by the equations 
$x_1=x_2=x'_1=x'_2=t=0$.   

The above blow up $\epsilon : {\Y} \to {\Y}'$ is a small blow up and can be 
described directly as follows:
we blow up ${\Y}'$ along its subvariety $\Pi$ defined by  
$x_1=x_2'=0$ (a 2-plane contained in the central fiber of ${\Y}'$).
The proper transform ${\Y}$ of ${\Y}'$ is smooth.
In local coordinates, the blow-up is given by the
graph $\Gamma_{\phi} \subseteq Y' \times \PP^1$
of the rational map $\phi: {\Y}' \longrightarrow {\PP}^1$ given by
$\phi(x_1,\ldots,x'_{g+1},t)=(x_1: x'_2)$. The equations of the graph
$\Gamma_{\phi}\subseteq Y' \times \PP^1 
\subseteq A^{2(g+1)}_S \times_S {\PP}^1_S$
are given by the system
$$
x_1x_2=t,\; ux_2'-vx_1=0,\; ux_2-vx_1'=0\, ,
$$
where $u,v$ are homogeneous coordinates on $\PP ^1$. 

\end{section}
%%%%%%%%%%%%%%%%%%%%%%%%%%%%%%%%%%%%%%%%%%%%%%%%%%%%%%%%%%%%%%%%%%%%%%%%%%%%%%%%%%%%%%%%%%%%
%%%%%%%%%%%%%%%%%%%%%%%%%%%%%%%%%%%%%%%%%%%%%%%%%%%%%%%%%%%%%%%%%%%%%%%%%%%%%%%%%%%%%%%%%%%%
\begin{section}{Extension of the Poincar\'e bundle}
\label{sec:Poincare}
We denote by $j_0:X_0\hookrightarrow {\X}$ and  
$i_0: Y_0\hookrightarrow {\Y}$ the inclusions of the special fiber.
Recall that we write $V$ for $Y_0$ and $\tilde{V}$ for its normalization.
We denote by ${\Pcal}_{\eta}$ the Poincar\'e bundle
on  $Y'_{\eta}$ and  by $P_B$ the Poincar\'e bundle on $B$.

\begin{theorem}\label{pb}
The Poincar\'e bundle ${\Pcal}_{\eta}$ has an extension
${\Pcal}$ such that the pull back of ${\Pcal}_0:=i_0^*{\Pcal}$ to 
$\tilde{V}$ satisfies
$\sigma^* \Pcal _0 \cong \tau^*(q \times q)^* P_B \otimes 
{\Ocal}(- E_{12}-E_{21})$.
\end{theorem}
\begin{proof}
We have the following commutative diagram of maps
\begin{displaymath}
\begin{xy}
\xymatrix{  
&V \ar[r]^{\bar{m}} & \overline{\PP} \\
& \tilde{V} \ar[u]_{\sigma}\ar[d]^{\tau} \ar[r]^{\tilde{m}} & {\PP}\ar@{=}[d]\ar[u]^{\nu} \\
{\PP}\ar[d]^q &{\PP}\times {\PP}\ar[d]_{q\times q} \ar[l]_{p_i} & {\PP}\ar[d]_q \\
B & B \times B \ar[l]_{q_i} \ar[r]^{\mu_B} & B
}
\end{xy}
\end{displaymath}

Let ${\Lcal}$ be the theta line bundle on the family 
${\X}$ introduced in section~\ref{sec:family}. 
We define the extension of ${\Pcal}^0$ by
$$
{\Pcal}:= \bar{\mu}^*{\Lcal}\otimes \rho_1^*{\Lcal}^{-1}\otimes 
\rho_2^*{\Lcal}^{-1},
$$
where we denote by  $\rho_1, \rho_2: {\Y} \rightarrow {\X}$ 
the compositions of the natural projections 
$\rho{'}_i: {\Y}' \rightarrow {\X}$ with the blowing up map
$\epsilon: {\Y} \rightarrow {\Y}'$ of section \ref{modelY}. We then have 
$\sigma^* \Pcal _0 =\sigma^*(\bar{m}^*j_0^*{\Lcal}) \otimes 
\sigma^*i_0^*\rho_1^*{\Lcal}^{-1}\otimes \sigma^*i_0^*\rho_2^*{\Lcal}^{-1}$.  
Now $\bar{m}^*j_0^*{\Lcal}=\bar{m}^*\bar{L}$, so 
$\sigma^*(\bar{m}^*j_0^* {\Lcal})=\sigma ^*\bar{m}^*\bar{L}
=\tilde{m}^* \nu^*\bar{L}=\tilde{m}^*({\Ocal}({\PP}_1) \otimes q^*M_b)$. 
In view of ${\Ocal}({\PP}_1)= {\Ocal}(1)$ 
we have $\tilde{m}^*  \Ocal ({\PP}_1)=N$, where $N$ is the line bundle introduced 
at the end of section \ref{sec:addition}. We thus get
$$
\tilde{m}^*{\Ocal}({\PP}_1)=\tau^*p_1^*{\Ocal}({\PP}_1) \otimes
\tau^*p_2^* \Ocal ({\PP}_1)\otimes \Ocal (-E_{12}-E_{21})
$$ 
and $\tilde{m}^*q^*M_b=\tau^*(q\times q)^*\mu_B^*M_b$.
On the other hand using the description of $\bar{L}$ in \S \ref{sec:family}
we see
$$
\begin{aligned}
\sigma^*(i_0^*\rho_i^* \Lcal )=&
\tau^* p_i^*\nu^* \bar{L} =\tau^*p_i^*( \Ocal ({\PP}_1) \otimes q^*M_b)\\
=&\tau^*p_i^* \Ocal ({\PP}_1) \otimes \tau^*(q\times q)^* q_i^*M_b.\\
\end{aligned}
$$ 
and putting this together we find
$$
\begin{aligned}
\sigma^* \Pcal _0&=\tau^*(q\times q)^*(\mu_B^*M_b\otimes q_1^*M_b^{-1}
\otimes q_2^*M_b^{-1}) 
\otimes \Ocal (-E_{12}-E_{21})\\
&=\tau^*(q\times q)^*P_B \otimes \Ocal (-E_{12}-E_{21}). \\
\end{aligned}
$$
\end{proof}
\end{section}
%%%%%%%%%%%%%%%%%%%%%%%%%%%%%%%%%%%%%%%%%%%%%%%%%%%%%%%%%%%%%%%%%%%%%%%%%%%%%%%%%%%%
%%%%%%%%%%%%%%%%%%%%%%%%%%%%%%%%%%%%%%%%%%%%%%%%%%%%%%%%%%%%%%%%%%%%%%%%%%%%%%%%%%%%
\begin{section}{The basic construction}
The fibration $\pi: {\Y}\rightarrow S$ is a flat map 
since $\Y $ is irreducible and $S$ is smooth $1$-dimensional, 
see  \cite{Ha}, Ch.\ III, Proposition 9.7.
The maps  $\rho_i=\Y \to \X $, $i=1,2$, defined in the proof of 
Theorem \ref{pb}, are flat maps too 
since they are  maps of smooth irreducible varieties with fibers of 
constant dimension $g$, see e.g.\ \cite{Mats}, Corollary of Thm.\ 23.1.

We denote by $Y_0$ (resp.\ $Y_{\eta}$) the special fibre (resp.\  
the generic fibre) 
and by $i_0: Y_0 \to{\Y}$ (resp.\  
$i_{\eta}: Y_{\eta} \to{\Y}$) the corresponding embedding.
According to \cite{F}, Example 10.1.2., $i_0$ 
is a regular embedding. Similarly, $j_0: X_0 \to \X$ is  a regular embedding. We consider the diagram
\begin{displaymath}
\begin{xy}
\xymatrix{
Y_0 \ar[r]^{i_0} \ar[d]^{\pi_0}& {\mathcal Y}\ar[d]^{\pi} \\
{\rm Spec}(k) \ar[r]^{s} & S \\
}
\end{xy}
\end{displaymath}
 Let $i_0^*: A_k(\Y)\to A_{k-1}({Y_{0}})$ be the Gysin map
(see  \cite{F}, Example 5.2.1).
Since $Y_0$ is an  effective Cartier divisor in $\Y$ the  Gysin map $i_0^*$
coincides with the Gysin map for divisors (see 
\cite{F}, Example 5.2.1 (a) and $\S \, 2.6$).

We now consider specialization of cycles, see \cite{F}, $\S \, 20.3$. 
Note that according to \cite{F}, Remark 6.2.1., in our case we have  
$s^{!} a = i_0^* a,\; a \in A_*(\Y)$.   
If ${\mathcal Z}$ is a flat scheme over the spectrum of a discrete valuation ring $S$ the specialization homomorphism $\sigma_Z: A_k(Z_{\eta})\to A_k(Z_0)$ is defined as follows, see  \cite{F}, pg. 399: If 
$\beta_{\eta}$ is a cycle on $Z_{\eta}$ 
we denote by $\beta$ an extension of $\beta_{\eta}$ in ${\mathcal Z}$ 
(e.g.\  the Zariski closure of $\beta_{\eta}$ in ${\mathcal Z}$) and then  
$\sigma_Z(\beta_{\eta})=i_0^*(\beta)$, 
where $i_0: Z_0 \to {\mathcal Z}$ is the natural embedding.
 
Let $c_{\eta}$ be a cycle on $X_{\eta}$ and let 
$\varphi_{\eta}=F(c_{\eta})$ be the Fourier-Mukai transform. 
It is defined by
$F(c_{\eta})=\rho_{2*}(e^{c_1({\Pcal}_{\eta})}\cdot \rho_1^*c_{\eta} ) 
\in A_*(X_{\eta})$. 
Let $\sigma_X :A_k(X_{\eta})\to A_k(X_0)$ be the specialization map. 
We have to determine $\sigma_X(F(c_{\eta}))$. 

If $\beta_{\eta}$ is a cycle on  $A_k(Y_{\eta})$ we have 
$\rho_{2*} \sigma _{Y} (\beta_{\eta}) = \sigma_X \rho_{2*}(\beta_{\eta})$ 
by applying \cite{F} Proposition 20.3 (a) to the proper map 
$\rho_{2}: {\Y}\rightarrow {\X}$.
By choosing
$\beta_{\eta}= e^{c_1({\Pcal}_{\eta})}\cdot \rho_1^{*}c_{\eta}$  we have

\begin{equation}\label{eq:1}
\sigma_X (F(c_{\eta}))=\rho_{2*} \sigma_{Y} 
(e^{c_1({\Pcal}_{\eta})}\cdot \rho_1^{*}c_{\eta}) \; .
\end{equation}

Therefore, in order to compute
$\sigma_X ({\Fcal}(c_{\eta}))$ we have to identify  
$\sigma_{Y}  (e^{c_1({\Pcal}_{\eta})} \cdot \rho_1^{*}c_{\eta})$.
We take the extension $e^{c_1({\Pcal})}$ of $e^{c_1({\Pcal}_{\eta})}$ 
and the  extension of
$\rho_1^{*}c_{\eta}$ given  by  $\rho_1^*c$, where $c$ is the Zariski closure of $c_{\eta}$ in $\X$.
Since $i_{\eta}:Y_{\eta} \to \Y$ is an open embedding 
and hence a flat map of dimension $0$, we have $i_{\eta}^*(e^{c_1({\Pcal})}\cdot  \rho_1^{*}c)=e^{c_1({\Pcal}_{\eta})}\cdot \rho_1^{*}c_{\eta} $, see \cite{F}, Proposition 2.3 (d).
In other words, the cycle $e^{c_1({\Pcal})}\cdot  \rho_1^{*}c$ extends the cycle 
$e^{c_1({\Pcal}_{\eta})}\cdot \rho_1^{*}c_{\eta}$ and 
hence 
$\sigma_{Y}  (e^{c_1({\Pcal}_{\eta})} \cdot \rho_1^{*}c_{\eta})=i_0^*(e^{c_1({\Pcal})}\cdot  \rho_1^{*}c)$.

Now, for any $k$-cycle $a$ on $\Y$ we have the identity
$$
i_0^*(c_1({\Pcal}) \cdot a)= 
c_1(\Pcal_0) \cdot i^*_0(a)
$$
in $A_{k-2}(Y_0)$, where $\Pcal_0=i^*_0{\Pcal}$ 
is the pull back of the line bundle and $i_0^*a$ 
the Gysin pull back to the divisor $Y_{0}$.
This follows from applying the  formula in  \cite{F}, Proposition 2.6 (e)
to $i_0: Y_0\to \Y $, with $D=Y_0$, 
$X=\Y$ and $L={\Pcal}$ the Poincar\'e bundle.
Hence
\begin{equation}
\label{eq:2}\sigma_{Y} (e^{c_1({\Pcal}_{\eta})}\cdot \rho_1^{*}c_{\eta}) = 
e^{c_1({\Pcal}_0)}\cdot  i^*_0(\rho_1^*c) \; .
\end{equation}
By the Moving Lemma (see \cite{F}, $\S 11.4$), we may choose the cycle 
$c$ on the regular ${\X}$ such that it intersects the singular locus $T$ 
of the central fiber properly.  
Since $T \subseteq X_0$ the cycle $c_0=j_0^*(c)$ meets $T$ properly by the following 
dimension argument. We have
${\rm dim}(c \cap T)={\rm dim}(c_0 \cap T)$, 
hence
$$
\begin{aligned}
{\rm dim}(c_0 \cap T) &= {\rm dim}(c)+{\rm dim}(T)-{\rm dim}(X)\\
&=({\rm dim}(c)-1)+{\rm dim}(T)
 -({\rm dim}(X)-1)\\
&={\rm dim}(c_0)+{\rm dim}(T)-{\rm dim}(X_0).\\
\end{aligned} 
$$
Since $T$ is of  codimension $1$ in $X_0=\bar{\PP}$, saying that $c_0$ 
meets $T$ properly, is equivalent to saying that no component of $c_0$ 
is contained in $T$.
\begin{lemma} \label{le:onto}
There exists a cycle $\gamma$ on ${\PP}$ with $c_0=\nu_* \gamma $ 
that meets the sections ${\PP}_i$ for $i=1,2$ properly.
\end{lemma}
\begin{proof}
If $T$ is the singular locus of $\bar{\PP}$ and $A={\PP}_1\cup {\PP}_2$ 
its preimage in ${\PP}$, then $\bar{\PP}\backslash T \cong {\PP} \backslash A$. 
We may assume that the cycle $c_0$ is irreducible and we consider  
the support of $c_0\cap(\bar{\PP}\backslash T)$ as a subset $W$ of 
${\PP}\backslash A$.  Its Zariski closure $\gamma =\bar{W}$ is an 
irreducible cycle on ${\PP}$. Then $\nu_*\gamma$ is an irreducible cycle 
on $\bar{\PP}$ since  the map $\nu$ is a projective map.  Also,
$\nu_*\gamma \cap (\bar{\PP}\backslash T)= c_0 \cap(\bar{\PP}\backslash T)$,
hence $\nu_*\gamma $ is the Zariski closure of $c_0\cap(\bar{\PP}\backslash T)$
and so, by the irreducibility, we have $\nu_*\gamma=c_0$.
\end{proof}

\begin{lemma}\label{le:compair}
If $c_0=\nu_*\gamma$, then we have
$i_0^* \rho_1^*c=\sigma_*(\tau^*(p_1^*\gamma))$.
\end{lemma}
\begin{proof}
We denote the restriction of $\rho_i$ to the special fibre again by $\rho_i$. 
Then we have $i_0^* \rho_1^*c=\rho_1^*c_0$ since $\rho_1$ is a flat map and $i_0, j_0$ are regular 
embeddings (see \cite{F}, Theorem 6.2 (b) and Remark 6.2.1).
We will use the following commutative diagram
\begin{displaymath}
\begin{xy}
\xymatrix{  
&   \tilde{V} \ar[dl]_{\tau}\ar[dd] \ar[r]^{\sigma} & V \ar[dr]^{\epsilon}\ar[dd]_{\rho_i} & \\
{\PP} \times {\PP}  \ar[dr]_{p_i} & &  & \overline{\PP}\times \overline{\PP} \ar[dl]^{\rho^{\prime}_i} \\
& {\PP}\ar[r]^{\nu} & \overline{\PP} &  & \\
}
\end{xy}
\end{displaymath}
We may assume that $c_0$  and $\gamma $ are irreducible $k$-cycles.  
We claim that $\rho_1^*c_0$ is irreducible. Indeed, the map $\rho_1$ is a flat 
map of relative dimension $g$. The cycle $\rho_1^*c_0$ is then a cycle of 
pure dimension $k+g$ and contains the proper transform of 
$(\rho^{\prime}_1)^*c_0$ 
and that is an irreducible cycle. Any other irreducible component of  
$\rho_1^*c_0$ must have support on the preimage of  $T$. But since the cycle $c_0$ intersects 
$T$ along a $k-1$-cycle, there is no irreducible component of $\rho_1^*c_0$ 
on the preimage of $T$. 
On the other hand, since $\gamma $ meets the sections $\PP _i$ properly,
the cycle $\tau^*p_1^*\gamma$ is an irreducible cycle, and hence so is 
$\sigma_*(\tau^*p_1^*\gamma) $. But as $\rho_1^*c_0$ and 
$\sigma_*(\tau^*p_1^*\gamma) $ coincide outside the exceptional divisor of $V$, 
they have to coincide everywhere.
\end{proof}

\begin{proposition}
\label{prop:specialization} We have
$\sigma_X({\Fcal}(c_{\eta}))=\rho_{2*}(e^{c_1({\Pcal}_0)} \cdot 
\sigma_*(\tau^*p_1^*\gamma))$.
\end{proposition}
\begin{proof}
By equation (\ref{eq:2}) and Lemma \ref{le:compair} we have
\begin{equation}
\label{eq:4}
 \sigma _{Y} (e^{c_1({\Pcal}_{\eta})}\cdot \rho_1^{*}c_{\eta}) = 
e^{c_1({\Pcal}_0)}
\cdot  \sigma_*\tau^*(p_1^*\gamma )   \; .
\end{equation}
The result follows from equation (\ref{eq:1}).
\end{proof}
In order to calculate the limit of the Fourier-Mukai transform we are thus reduced to a calculation in the special fibre.
\end{section} 
%%%%%%%%%%%%%%%%%%%%%%%%%%%%%%%%%%%%%%%%%%%%%%%%%%%%%%%%%%%%%%%%%%%%%%%%%%%%%%%%%%%%%%%%
%%%%%%%%%%%%%%%%%%%%%%%%%%%%%%%%%%%%%%%%%%%%%%%%%%%%%%%%%%%%%%%%%%%%%%%%%%%%%%%%%%%%%%%%
\begin{section}{A calculation in the special fibre - Proof of the main theorem}
\label{specialfiber}

Recall the normalization map $\sigma:\tilde{V} \to V$. Suppose we have
a cycle $\rho$ on $\tilde{V}$ with $\sigma_*{\rho}=c_0$. We can consider
the intersection $c_1({\Pcal}_0)^k \cdot c_0$, that is a successive 
intersection of a cycle with a Cartier divisor on the singular variety $V$.
On the other hand we have 
the cycle $\sigma_*(c_1(\sigma^*{\Pcal}_0)^k \cdot \rho)$
and the projection formula (\cite{F}, Proposition 2.5 (c)) implies that
$$
c_1({\Pcal}_0)^k \cdot c_0=\sigma_*(c_1(\sigma^*{\Pcal}_0)^k \cdot \rho).
$$
Now we will use the following diagram of maps.
\begin{displaymath}
\begin{xy}
\xymatrix{  
&&& \tilde{V}\ar[r]^{\sigma}\ar[d]^{\tau} &V \\
&&& {\PP}\times {\PP}\ar[dll]_{p_1} \ar[dl]^{\alpha_2} \ar[dd]^{q\times q}
 \ar[dr]_{\alpha_1}
\ar[drr]^{p_2} \\
\overline{\PP} & {\PP}\ar[dr]_q \ar[l]^{\nu}  &{\PP}\times B 
\ar[dr]^{\beta_1}\ar[l]^{\kappa_1}&& B \times {\PP}\ar[dl]_{\beta_2} \ar[r]_{\kappa_2} & {\PP} \ar[dl]^q \ar[r]_{\nu} & \overline{\PP} \\
&&B  & B\times B\ar[l]^{q_1} \ar[r]_{q_2} & B \\
}
\end{xy}
\end{displaymath}

\begin{lemma}\label{ale:basicinters}
Let $x$ be a cycle on $B\times B$.   Then the following holds. 
\begin{enumerate}
\item  $p_{2*}((q\times q)^*x)=0$.  
\item  $p_{2*}( (q\times q)^*x \cdot p_1^*\eta)= q^*q_{2*}x $.
\end{enumerate}
\end{lemma}
\begin{proof} 
For (1) we observe that  $p_{2*}=\kappa_{2*} \alpha_{1*}$,
and $(q\times q)^*= \alpha_1^* \beta_2^*$ and $\alpha_{1*} \alpha_1^*=0$.
For (2) we use the identities  
$$
\begin{aligned}
p_{2*}((q\times q)^*x \cdot p_1^*\eta)
&=p_{2*}(\alpha_2^*\beta_1^* x \cdot \alpha_2^* \kappa_1^*\eta)
=p_{2*}\alpha_2^*(\beta_1^*x \cdot \kappa_1^*\eta)\\
&= \kappa_{2*}\alpha_{1*}\alpha_2^*(\beta_1^*x \cdot \kappa_1^*\eta)= 
\kappa_{2*}\beta_2^*\beta_{1*}(\beta_1^*x \cdot \kappa_1^*\eta)\\
&=\kappa_{2*}\beta_2^*(x \cdot \beta_{1*}\kappa_1^*\eta)
=q^*q_{2*}(x \cdot q_1^*q_*\eta)= q^*q_{2*}x.\\
\end{aligned}
$$
\end{proof}
Consider the following diagram of maps
\begin{displaymath}
\begin{xy}
\xymatrix{
{\PP}_i \ar[d]_{\lambda_i} & {\PP}_i \times {\PP}_j\ar[d]_{\lambda_{ij}} 
& E_{ij} \ar[d]_{\epsilon_{ij}}\ar[l]_{\pi_{ij}} \\
{\PP}\ar[d]_q &   {\PP}\times {\PP}\ar[l]_{p_1} \ar[d]_{q\times q} & \tilde{V} \ar[l]_{\tau}\ar[d]_{\sigma} \\
B & B \times B \ar[l]_{q_1} \ar[d]_{q_2} & V\\
 & B \\
}
\end{xy}
\end{displaymath}
where $p_i,\,q_i$ are the projections to the $i$th factor, $\pi_{ij}$ 
the canonical map of the projective bundle $E_{ij}$ and  the maps
$\lambda _i$,  $\lambda_{ij}$ and $\epsilon_{ij}$ the natural inclusions.
The map $(q\times q) \circ \lambda_{ij}$ is an isomorphism.

By the adjunction formula, the normal bundles to 
${\PP}_1,\, {\PP}_2$    are $N_{{\PP}_1}({\PP})= J $ and $N_{{\PP}_2}({\PP})=J^{-1}$. 
The exceptional divisors $E_{12}$ and $E_{21}$ are 
projective bundles over the blowing up loci ${\PP}_i \times {\PP}_j$.  
By identifying  $\PP_i \times \PP_j$ with $B \times B$, via the map $(q \times q) \circ \lambda_{ij}$, 
 we  have $E_{12}=\PP(q_1^*J^{-1}\oplus q_2^*J)$ and $E_{21}=\PP(q_1^*J\oplus q_2^*J^{-1})$.
We set $\xi_{ij} =c_1(O(1))$ on $E_{ij}$. By standard theory [\cite{Ha}, ch.\ II, Theorem 8.24 (c)] 
we have $\epsilon_{ij}^*E_{ij}=-\xi_{ij}$. 

We now introduce the notation 
$$
\gamma:=c_1(J),\;\; \gamma_i=q_i^*\gamma,\;\; \eta_i=p_i^*\eta,\;\; i=1,2.
$$
Note that $\gamma $  is algebraically equivalent to $0$, but not rationally
equivalent to~$0$. We have the quadratic relations
$$
\xi_{ij}^2 + {\pi^{\prime}}^*_{ij}(\gamma_i - \gamma_j )\cdot \xi_{ij} -
{\pi^{\prime}}^*_{ij} (\gamma_i \gamma_j) =0,
$$
where $\pi^{\prime}_{ij}: E_{ij} \to B \times B$ is the natural map.

\begin{lemma}\label{xiphilemma}
Suppose that $\xi$ satisfies the relation  $\xi^2+(a-b)\xi-ab=0$. Then, with
$\phi_{k}=\sum_{m=0}^{k-1} (-1)^m a^mb^{k-1-m}$ we have
$\xi^k=\phi_k\xi +ab\phi_{k-1} $ for any $k \geq 1$ (where we put
$\phi_0=0$).
\end{lemma}
\begin{proof} 
Assuming by induction that $\xi^k=\phi_k\xi +ab\phi_{k-1}$
we find  
$$
\xi^{k+1}=\phi_k\xi^2+ab\phi_{k-1}\xi=((b-a)\phi_k+ab\phi_{k-1})\xi+ ab\phi_k,
$$
so the result follows by induction from the recurrence
$\phi_{k+1}=(b-a)\phi_k+ab\phi_{k-1}$
that can be left to the reader.
\end{proof}

Applying the above for the classes $\xi_{ij}$ of the bundles 
$E_{ij}$, considered as bundles over 
$B \times B$ via the isomorphism $(q \times q) \circ \lambda_{ij}$, 
we get, by choosing 
$$
\phi_k = \sum_{m=0}^{k-1} (-1)^m \gamma_1^m \gamma_2^{k-1-m}, 
$$
that
$$
\begin{aligned}
\xi_{12}^k=& {\pi^{\prime}}^*_{12} \phi_k \cdot \xi_{12} \, + 
{\pi^{\prime}}^*_{12}( \gamma_1  \gamma_2 
\phi_{k-1}), \\  
\xi_{21}^k=&(-1)^{k+1} {\pi^{\prime}}^*_{21} \phi_k \cdot \xi_{21}  
\,+\, (-1)^k {\pi^{\prime}}^*_{21}(\gamma_1 \gamma_2 \phi_{k-1})\, .
\end{aligned}
$$ 

We view now the bundles $E_{ij}$ as bundles over  
${\PP}_i \times {\PP}_j$ and, for any $k \geq 0$, we write 
$\xi_{ij}^k=\pi_{ij}^*A_{ij}(k)\xi_{ij}+\pi_{ij}^*B_{ij}(k)$, for some cycles 
$A_{ij}(k),\, B_{ij}(k)$ on $\PP_i \times \PP_j$. By the above relations
we have
$$
(q\times q)_* \lambda _{ij*} A_{ij}(k) =(-1)^{(k+1)j} \phi_k \, . 
$$

\begin{lemma}\label{ale:lambda}
We have
$$
\lambda_{ij*}A_{ij}(k)=(-1)^{(k+1)j}[ (q\times q)^*\phi_k \cdot  \eta_1
 \eta_2 - (q\times q)^*(\phi_k \, \gamma_j)  \cdot \eta_i ]\, .
$$
\end{lemma}
\begin{proof} We let $\psi_{ij}= (q \times q)\circ\lambda_{ij}: 
{\PP}_i \times {\PP}_j \to B\times B$ be the natural isomorphism. 
We then have the identity
$$
\lambda_{ij*}A_{ij}(k)= \lambda_{ij*}(\psi_{ij}^* \psi_{ij*}A_{ij}(k))=
(q\times q)^*\psi_{ij*}A_{ij}(k) \cdot 
\lambda_{ij*}1_{\PP_i \times \PP_j}.
$$ 
But $\lambda_{ij*}1_{\PP_i \times \PP_j}
= p_1^*\PP_i \cdot p_2^* \PP_j
= \eta_i( \eta_j- p_j^* q^* \gamma ) = \eta_1 \eta_2- \eta_i \cdot (q\times q)^*\gamma_j$ 
and the result follows.
\end{proof}

\begin{lemma} \label{lem:basic}
For a cycle class $x=q^*z +q^*w \cdot \eta$  on $\PP$ the
cycle  class $\tau_*(\tau^*p_1^*x \cdot (E^k_{12}+E_{21}^k))$
for $k \geq 1$ is given by
$$
\begin{aligned}
&\sum_{m=0}^{k-2} (-1)^m  \{(q\times q)^* q_1^*[(( (-1)^{k+1}-1)z +(-1)^{k+1}w 
\gamma)  \,  \gamma^{m}]  \cdot  \eta_1  \eta_2 \\
&+(-1)^k (q\times q)^*q_1^*[( z+w \gamma)\,  \gamma^{m}] \cdot  \eta_1 \cdot p_2^* q^*\gamma 
+(q\times q)^*q_1^*( z  \gamma^{m+1}) \cdot  \eta_2\}\cdot p_2^* q^* \gamma^{k-2-m}. \\
\end{aligned}
$$
\end{lemma}
Note that for $k=1$ the above sum is zero. 

\begin{proof}
Since $\epsilon_{ij}^* E_{ij}=-\xi_{ij}$ 
we have  $E_{ij}^k=(-1)^{k-1} \epsilon_{ij*}\xi_{ij}^{k-1}$. 
Therefore
$$
\begin{aligned}
\tau_*(\tau^*p_1^*x \cdot E^k_{ij})=& 
(-1)^{k-1}p_1^*x \cdot \tau_* \epsilon_{ij*}\xi_{ij}^{k-1} \\
=&(-1)^{k-1} p_1^*x \cdot \lambda_{ij*}\pi_{ij*}(\pi_{ij}^*A_{ij}(k-1)\xi_{ij}+
\pi_{ij}^*B_{ij}(k-1)) \\
=&(-1)^{k-1} p_1^*x \cdot \lambda_{ij*}A_{ij}(k-1)\\
\end{aligned}
$$ 
since $\pi_{ij*}\xi_{ij}=1_{\PP_i \times \PP_j}$. Note that since $A_{ij}(0)=0$ the above calculation shows that 
$\tau_*(\tau^*p_1^*x \cdot E_{ij}) =0$.
By Lemma \ref{ale:lambda} 
and by using the relation 
$$
p_1^*x=(q\times q)^* q_1^* z +  (q\times q)^* q_1^* w \cdot \eta_1,
$$ 
we have
$$
\begin{aligned}
\tau_*(\tau^*p_1^*x \cdot E^k_{ij})= & (-1)^{k(j+1)+1} ((q\times q)^* q_1^* z +  (q\times q)^* q_1^* w \cdot \eta_1)  \\
& \cdot [(q\times q)^*\phi_{k-1} \cdot  \eta_1 \eta_2 - 
(q\times q)^*(\phi_{k-1} \gamma_j)  \cdot \eta_i ] \\
\end{aligned}
$$
and this equals
$$
\begin{aligned}
& (-1)^{k(j+1)+1} [(q\times q)^* (q_1^* z\cdot \phi_{k-1})\cdot  
\eta_1 \eta_2 - 
(q\times q)^*(q_1^* z \cdot \phi_{k-1} \gamma_j)  
\cdot \eta_i   \\
& + (q\times q)^* (q_1^*  w \cdot \phi_{k-1})
\cdot  \eta^2_1 \eta_2 -(q\times q)^*(q_1^*w\cdot 
\phi_{k-1} \gamma_j) \cdot  \eta_1 \eta_i ]  
\end{aligned}
$$
We then have, by using the formula $\eta^2= q^*\gamma \cdot \eta$,   that
$$                                                                             
\begin{aligned}
\tau_*(\tau^*p_1^*x \cdot E^k_{12})=  
  (-1)^{k+1} [(q\times q)^* (q_1^*( z+w \gamma) \cdot 
\phi_{k-1})\cdot  \eta_1 \eta_2 &\\
 - (q\times q)^*(q_1^*( z+w \gamma)\cdot  \phi_{k-1} ) 
\cdot  \eta_1\cdot  p_2^* q^*\gamma] & \\
\end{aligned}
$$
and
$$
\begin{aligned}
\tau_*(\tau^*p_1^*x \cdot E^k_{21})=  - & (q\times q)^* (q_1^*z \cdot \phi_{k-1})
\cdot  \eta_1 \eta_2 +
(q\times q)^*(q_1^*( z \, \gamma) \cdot \phi_{k-1} ) \cdot  \eta_2. \\
\end{aligned}
$$
Using 
$\phi_{k-1} = 
\sum_{m=0}^{k-2} (-1)^m \gamma_1^m \cdot \gamma_2^{k-2-m}  $ we deduce
the proposition.
\end{proof}

We state now the basic result of this section.

\begin{proposition} \label{aprop:FTq^*a}
Let $z,\, w $ be  cycles on $B$. Then we have
$$ 
p_{2*}\tau_*(e^{c_1 (\sigma^* \Pcal _0)} 
\cdot \tau^*(p_1^*(q^*z+q^*w \cdot \eta))=  q^*a+q^* b \cdot \eta,
$$ 
with $a$ and $b$ as in Theorem \ref{th: aeFT}.
\end{proposition}
\begin{proof} We put $x= q^*z+q^*w \cdot \eta$. 
We  want to calculate
$$
p_{2*}\tau_*(e^{\tau^*(q\times q)^*c_1(P_B) -E_{12}-E_{21}}\cdot 
\tau^*(p_1^*x))
$$
which equals
$$
p_{2*}(e^{(q\times q)^*c_1(P_B)}\cdot \tau_* ( e^{-E_{12}-E_{21}}\cdot 
\tau^*p_1^*x)).
$$
Since $E_{12}\cdot E_{21}=0$ we have 
$$
e^{-E_{12}-E_{21}}=1+ \sum_{k=1}^{2g} \frac{(-1)^k}{k!}\, (E_{12}^k + E_{21}^k)
$$
and so
$\tau_* ( e^{-E_{12}-E_{21}}\cdot \tau^*p_1^*x)$
equals
$$
p_1^*x+ \sum_{k=1}^{2g} \frac{(-1)^k}{k!}\, \tau_*[\tau^*p_1^*x \cdot (E_{12}^k + E_{21}^k)].
$$
We have
$$
\begin{aligned}
 p_{2*}((q\times q)^* e^{c_1(P_B)}\cdot  p_1^*x) &=  
p_{2*}(e^{(q \times q)^*c_1(P_B)} \cdot  p_1^*(q^*z+q^*w\, \eta))\\
& =p_{2*}((q \times q)^*(e^{c_1(P_B)} q_1^*z)+(q \times q)^*(e^{c_1(P_B)}q_1^*w)\, p_1^*\eta)\\
& = 0+q^*q_{2*}(e^{c_1(P_B)}q_1^*w)=  q^*F_B(w) \\
\end{aligned}
$$
by  Lemma \ref{ale:basicinters}. Combining the above with 
Lemma  \ref{lem:basic} we find that 
$$
p_{2*}\tau_*(e^{\tau^*(q\times q)^*c_1(P_B) -E_{12}-E_{21}}\cdot 
\tau^*(p_1^*x))
$$
is the sum of the four terms: the first is $q^*F_B(w)$, the second is 
$$
\begin{aligned}
\sum_{k=2}^{2g} \sum_{m=0}^{k-2} \frac{(-1)^{k+m}}{k!}\, 
\{ p_{2*}[(q\times q)^*[e^{c_1(P_B)} q_1^*[(( (-1)^{k+1}-1)z + &\\
(-1)^{k+1}w  \gamma)  \,  \gamma^{m}]]  \cdot \eta_1] \}\cdot \eta 
\cdot q^*\gamma^{k-2-m} , &\\
\end{aligned}
$$
the third term is
$$
\sum_{k=2}^{2g} \sum_{m=0}^{k-2} \frac{(-1)^{m}}{k!}\, \{ p_{2*}[(q\times q)^*[e^{c_1(P_B)}q_1^*[( z+w \gamma)\,  \gamma^{m}]] \cdot \eta_1]\} \cdot q^* \gamma^{k-1-m}   ,
$$
and finally the fourth is
$$
\sum_{k=2}^{2g} \sum_{m=0}^{k-2} \frac{(-1)^{k+m}}{k!}\, \{ p_{2*}[(q\times q)^*[e^{c_1(P_B)}q_1^*( z  \gamma^{m+1})]]\}\cdot  \eta  \cdot q^* \gamma^{k-2-m} \, . 
$$
By applying now Lemma \ref{ale:basicinters} and by making the substitution  $n=k-2$ we get 
the desired expression.
\end{proof}

\begin{corollary} \label{acor:FTS}
Let $z,\, w $ be  cycles on $B$. Then 
in algebraic equivalence we have 
$$  p_{2*}\tau_*(e^{c_1 (\sigma^* \Pcal _0)} 
\cdot \tau^*(p_1^*(q^*z+q^*w \cdot \eta))\a= \, q^* F_B(w)-q^*F_B(z) \cdot \eta.  
$$
\end{corollary}
\begin{proof}
Indeed,  since $c_1(J)\a= 0$  it is clear that 
$a \a= \, F_B(w)$ and $b \a= \, -q^*F_B(z)$
since the only non zero term of the sum  corresponds to $m=0,\; n=0$.
\end{proof}

We conclude now with the proof of the basic Theorem \ref{th: aeFT} and Theorem \ref{cor: aeFT}:
\begin{proof}
By Proposition \ref{prop:specialization}  we have 
$ \varphi_0=\sigma_X F(c_{\eta})=
\rho_{2*}(e^{c_1( \Pcal _0)} \cdot \sigma_*(\tau^* p_1^*\gamma))$. By the projection formula
 we have $e^{c_1( \Pcal _0)} \cdot \sigma_*(\tau ^* p_1^*\gamma) = 
\sigma_*(e^{c_1(\sigma^* \Pcal _0)} \cdot \tau ^* p_1^*\gamma)$. Observe now that 
$\rho_2 \circ \sigma = \nu \circ (p_2 \circ \tau ) : \tilde{V} \rightarrow\bar{\PP}$, see the diagram in the proof of Lemma \ref{le:compair}.  The proof then follows  
from Proposition \ref{aprop:FTq^*a} and Corollary \ref{acor:FTS}.
\end{proof}
\end{section}
%%%%%%%%%%%%%%%%%%%%%%%%%%%%%%%%%%%%%%%%%%%%%%%%%%%%%%%%%%%%%%%%%%%%%%%%%%%%%%%%%%%%%%%%%%%%
%%%%%%%%%%%%%%%%%%%%%%%%%%%%%%%%%%%%%%%%%%%%%%%%%%%%%%%%%%%%%%%%%%%%%%%%%%%%%%%%%%%%%%%%%%%%
\begin{section}{Applications}\label{Applications}
Let ${\X} \to S$ be a completed rank-one degeneration as described in \S 2.
According to Beauville \cite{B1} we have a decomposition of 
$CH^i_{\QQ}(X_{\eta})$
into subspaces which are eigenspaces for the action of the integers
on $X_{\eta}$:
$$
A^i_{\QQ}(X_{\eta}) =\oplus_j A^i_{(j)}(X_{\eta})
$$
such that $n^*(x)= n^{2i-j}\, x$ for $x \in A^i(X_{\eta})$.
(Beauville works over ${\CC}$, but his proof does not use more
than the Fourier-Mukai transform which works over the residue
field of $\eta$.)
The multiplication map $n$ acts as multiplication  by $n^{2i}$ on 
homology and therefore all cycles in $A^i_{(j)}(X_{\eta})$ are 
homologically trivial for $j \neq 0$. 
Since under the Fourier-Mukai transform we have 
$F(A^i_{(j)}(X_{\eta}))=A^{g-i+j}_{(j)}(X_{\eta})$, the elements
of $A^i$ that lie in $A^i_{(j)}$ can be characterized by their codimension 
(namely $g-i+j$). 

Suppose now that $c=\sum c^{(j)}\in A^i(X_{\eta})$
with $c^{(j)} \in A^i_{(j)}(X_{\eta})$, where the decomposition corresponds to
$\varphi:=F(c)=\sum \varphi^{(j)}$ with $\varphi^{(j)} \in A^{g-i+j}(X_{\eta})$.

\begin{theorem}
\label{nonzero}
Let $c=c_{\eta}=\sum c^{(j)} \in A^i(X_{\eta})$ with $c^{(j)}\in A^i_{(j)}(X_{\eta})$
such that $\varphi_0^{(j)}\neq 0$, where
$\varphi_0$ is the specialization and $\varphi_0^{(j)}$ the codimension
$g-i+j$-part of $\varphi_0$. Then $c^{(j)} \neq 0$.
\end{theorem}
\begin{proof}
The specialization map preserves the codimension of cycles. 
Therefore, if $c^{(j)}=0$ then $\varphi^{(j)}=0$, hence $\varphi_0^{(j)}=0$ 
and this contradicts our assumption.
\end{proof}

This theorem, which holds as well for cycles modulo algebraic equivalence,  
can be used to prove non-vanishing results for cycles. 
For the rest of this section 
we work modulo algebraic equivalence.  
For example, consider a threefold $\Z /S$ such that $Z_{\eta}$ is a smooth 
cubic threefold and $Z_0$ is a generic nodal cubic threefold. The genericity
assumption means that the corresponding canonical 
genus $4$ curve $C$ in ${\PP}^3$ 
which is used to construct the Fano threefold, see e.g.\ \cite{GK} Section 2, 
is a generic curve and hence we may assume by Ceresa's result \cite{C} 
that the class   $C^{(1)}$ does not vanish in the Jacobian $B$ of the 
curve $C$. Since $C$ is a trigonal curve we have by \cite{CvG} 
that $C^{(j)}\a= \,  0$  for $j \geq 2$. Hence the Beauville decomposition of 
$C$ is $[C] \a= \,  C^{(0)}+C^{(1)}$ with $F_B(C^{(0)}) \in A_{(0)}^1(B)$ and 
 $F_B(C^{(1)} )\in A_{(1)}^2(B)$.
 
The Picard variety $\X/S$ of $\Z$ defines a principally polarized  
semi-abelian variety with central fibre a rank-one extension of the Jacobian $B$ of the curve $C$, 
see \cite{GK}, Corollary 6.3 and Section 10.  The principal polarization on $X_{\eta}$ is induced by a geometrically defined divisor $\Theta$. Let  $\Sigma$ be the Fano surface of 
lines in $Z_{\eta}$. If $s \in \Sigma$ we  denote by $l_s$ the corresponding 
line in $Z_{\eta}$.  For each  $s \in S$ we have the divisor 
$$
D_s=\{ s'\in S, \; l_{s'} \cap l_s\neq \emptyset \}
$$ 
on $S$ as defined in \cite{CG}. We then  have a natural map
$$
\Sigma \to {\rm Pic}^0(\Sigma), \qquad s \mapsto D_s-D_{s_0},
$$
with $s_0\in \Sigma$ a base point. It is well known that the cohomology class  of 
$\Sigma$ in ${\rm Pic}^0(\Sigma)$ is equal to that of the cycle  ${\Theta ^3/ 3!}$, see \cite{CG}.   
By \cite{B1}, Propositions 3 and 4, we have that 
$A^3_{(j)}(X_{\eta})=0$ for $j <0$ and  $A^5_{(j)}(X_{\eta})=0$ for 
$j \neq 0$ in algebraic equivalence.
We have therefore the decomposition 
$$
[\Sigma]\a= \, \Sigma^{(0)}+\Sigma^{(1)}+\Sigma^{(2)} \qquad
\text{\rm with $\Sigma^{(j)} \in A^3_{(j)}$}.
$$ 
Indeed, $\Sigma ^{(j)} \in A^{3}_{(j)}(X_{\eta})$,
hence $F(\Sigma^{(j)})\in A^{2+j}_{(j)}(X_{\eta})$ 
which is zero for $j \geq 3$. Now we show that 
$\Sigma ^{(1)} \na= \,  0$, and we thus obtain a cycle 
which is homologically but not algebraically  
equivalent to zero. Since $\Theta  \in A^1_{(0)}(X_{\eta})$ this implies  that 
$\Sigma $ is homologically, but not algebraically   
equivalent to ${\Theta ^3 / 3!}$ .  

We denote by $\Xcal$ the completed rank one degeneration of 
$X_{\eta}$. The class $[\Sigma]$ degenerates to 
a cycle $[\Sigma_0]=\nu_*(\gamma )$ on the central fiber $X_0$ of class 
$$
\gamma \a= \,  q^*[C]+\frac{1}{2} \,q^*[C \ast C] \cdot \eta,
$$
where $C\ast C$ is the Pontryagin product, see \cite{GK}, 
Propositions 10.1 and 8.1.
In order to see that $\Sigma ^{(1)} \na= \,  0$ 
it suffices by Theorem \ref{nonzero} to show that $\varphi_0^{(1)} \na= \,  0$
with $\varphi_0$ the limit of the Fourier-Mukai transform.
By Theorem \ref{cor: aeFT}, we have
$$
\varphi_0\a= \, \nu_*(\frac{1}{2} \,q^*[F_B(C) \cdot F_B(C)]-q^*F_B(C) \cdot \eta),
$$
hence 
$$
\varphi_0^{(1)}\a= \, \nu_*(q^*[F_B(C^{(0)}) \cdot F_B(C^{(1)})]-q^*F_B(C^{(1)})\cdot \eta).
$$ 
Since $C^{(1)} \na= \,  0$  we conclude that  $\varphi_0^{(1)}\na= \,  0$,
and this implies the result.

By using the specialization of the Fourier-Mukai transform we can 
deduce the specialization of the Beauville decomposition. 
We do this working modulo algebraic equivalence.
 
\begin{proposition} \label{prop: aelimitBeauville}
Let $c=c_{\eta} \in A^i(X_{\eta})$ with specialization 
$c_0=\nu_*(q^*z+q^*w \cdot \eta )$, where $z \in A^i(B)$ and  
$w \in A^{i-1}(B)$. 
Let $c=  \sum c^{(j)}$ with $c^{(j)}\in A^i_{(j)}(X_{\eta})$,
and let $z=\sum z^{(j)}$ with $z^{(j)} \in A^i_{(j)}(B)$ and 
$w=\sum w^{(j)}$ with $w^{(j)} \in A^{i-1}_{(j)}(B)$  
be the Beauville decompositions. 
If $c_0^{(j)} $ is the specialization of $c^{(j)}$, then
$$  
c_0^{(j)}\a= \; \nu_*( q^*z^{(j)} +q^*w^{(j)} \cdot \eta )\,. 
$$
\end{proposition}
\begin{proof} 
By the proof of the main theorem in \cite{B1}, the 
component $c^{(j)}$ is defined as 
$(-1)^g F((-1)^*\phi^{(j)})$ with
$\phi^{(j)} \in A^{g-i+j}(X_{\eta})$ (notation as above).  
The inversion on $X_{\eta}$ leaves 
the cell decomposition of the toroidal compactification 
invariant and hence extends naturally to $X_0$.  
So  $c_0^{(j)}$ equals 
$(-1)^g F((-1)^*\phi_0^{(j)})$ with $\phi_0^{(j)} \in A^{g-i+j}(X_0)$. 
Therefore, by Theorem \ref{cor: aeFT}, 
we have 
 $$ 
 \begin{aligned} 
 c_0^{(j)} & \a= \; (-1)^gF((-1)^*\nu_*(q^*F_B(w^{(j)}) - q^*F_B(z^{(j)})\cdot \eta)) \\
 &\a= \; (-1)^{g+j} (-1)^{g-1+j} \nu_*(-q^*z^{(j)}  - q^*w^{(j)}\cdot\eta)=\nu_*( q^*z^{(j)} +q^*w^{(j)}
 \cdot \eta )\,.
 \end{aligned}
 $$
\end{proof}
For example, let $\C \to S$ be a genus $g$ curve with $C_{\eta}$ a smooth curve and $C_0$ a one-nodal curve with normalization $\tilde{C}_0$. Let $p$ be the node of $C_0$ and $x_1,\; x_2$ the points of $\tilde{C}_0$ lying 
over $p$.
 The compactified Jacobian $\X= \overline{P_{\C / S}} $ is then a complete rank one degeneration with central  fiber the $\PP ^1 $-bundle over ${\rm Pic}^0(\tilde{C}_0)$ associated to the line bundle $J=O(x_1-x_2)$. 
 Let $\bar{u}: \C \to \X $ be the compactified Abel-Jacobi map and let $c_{\eta}= [\bar{u}(C_{\eta})]$.
  The  cycle $c_{\eta}$ specializes then to the cycle $c_0=[\bar{u}(C_0)]$ with  
 $c_0 \a= \, \nu_*( q^*[{\rm pt}] + q^* \tilde{c}_0 \cdot \eta)$,  where $[{\rm pt}]$ is the class of a point and $\tilde{c}_0$ is  the class of the Abel-Jacobi image of the smooth curve $\tilde{C}_0$ in  
 ${\rm Pic}^0 (\tilde{C}_0)$, see e.g. \cite{GK}, Proposition 7.1. By Proposition \ref{prop: aelimitBeauville} we have then
 $$c_0^{(j)}\, \a= \; \begin{cases}  \; q^* \tilde{c}^{(j)}_0 \cdot \eta, \;\;  j \neq 0\, , \\ 
 \; q^* [{\rm pt}] +  q^* \tilde{c}^{(0)}_0 \cdot \eta, \;\; j=0\, . \end{cases}
  $$  
%%%%%%%%%%%%%%%%%%%%%%%%%%%%%%%%%%%%%%%%%%%%%%%%%%%%%%%%%%%%%%%%%%%%%%%%%%%%%%%%
{\bf Acknowledgement} The second author thanks the Korteweg-de Vries 
Instituut van de Universiteit van Amsterdam,  where part of this work was 
done, for its support and hospitality.
\end{section}
%%%%%%%%%%%%%%%%%%%%%%%%%%%%%%%%%%%%%%%%%%%%%%%%%%%%%%%%%%%%%%%%%%%%%%%%%%%%%%%%
%%%%%%%%%%%%%%%%%%%%%%%%%%%%%%%%%%%%%%%%%%%%%%%%%%%%%%%%%%%%%%%%%%%%%%%%%%%%%%%%


\begin{thebibliography}{99}
\bibitem{AN} V.\ Alexeev, I.\ Nakamura: On Mumford's construction
of degenerating abelian varieties.
{\sl Tohoku Math.\ J. \bf 51} (1999), p.\ 399--420.

\bibitem{B1} A.\ Beauville: Sur l'anneau de Chow d'une vari\'et\'e
ab\'elienne. {\sl Math.\ Annalen \bf 273} (1986), 647--651.

\bibitem{B2} A.\ Beauville: Algebraic Cycles on Jacobian Varieties.
{\sl Compositio Math.\ \bf 140} (2004), 683--688.

\bibitem{C} G.\ Ceresa: 
$C$ is not algebraically equivalent to $C^{-}$ in its Jacobian. 
{\sl Ann.\ of Math.\  \bf 117} (1983), 285--291.

\bibitem{CG} H.\ Clemens, Ph.\ Griffiths: 
The intermediate Jacobian of the Cubic Threefold.
{\sl Ann.\ Math.\ \bf 95} (1972), 281--356. 

\bibitem{CvG} E.\ Colombo, B.\ van Geemen: Notes on curves in a Jacobian.
{\sl Compositio Math.\ \bf 88} (1993), 333--353.

\bibitem{FC} G.\ Faltings, C.L. Chai: Degeneration of abelian varieties.
Ergebnisse der Mathematik, Springer Verlag.

\bibitem{F} W.\ Fulton: Intersection Theory. Ergebnisse der Mathematik.
3.\ Folge, {\bf 2}, 1998, Berlin, Springer Verlag.

\bibitem{GK} G.\ van der Geer, A.\ Kouvidakis: 
A Note on Fano Surfaces of Nodal Cubic Threefolds,
{\sl arXiv 0902.3877v1}. 

\bibitem{Ha} R.\ Hartshorne:  Algebraic Geometry. 
Graduate Texts in Math, Springer Verlag 1977.   

\bibitem{KKMS} G.\ Kempf, F.\ Knudsen, D.\ Mumford, B.\ Saint-Donat: 
Toroidal embeddings I, 
{\sl Lecture Notes in Math.\ \bf 339}, Springer Verlag 1972.

\bibitem{Mats}
H.\ Matsumura:  Commutative ring theory. Cambridge University Press, 1986.

\bibitem{Mumford1} D.\ Mumford: 
An analytic construction of degenerating abelian varieties over complete rings. 
{\sl Comp.\ Math.\ \bf 24}, (1972), 239--272.

\bibitem{Mu} D.\ Mumford: 
On the Kodaira dimension of the Siegel modular variety. In Open Problems
in Geometry.                     
{\sl Lecture Notes in Math.\  \bf 997}, 348--375. Springer Verlag 1983.

\bibitem{Na} Y.\ Namikawa: 
A new compactification of the Siegel space and degeneration of 
Abelian varieties, I, II.  
{\sl Math.\ Ann.\  \bf 221}, (1976), 97--141, 201--241.

\end{thebibliography}
 \end{document}